\documentclass[10pt, reqno, fleqn]{amsart} %\reqno = numero equations a droite

\usepackage{amsmath,amscd,amssymb,amsfonts,amsthm,courier,relsize,bm}
\usepackage{hyperref,enumitem,mathrsfs,slashed,mathtools, extarrows}
\usepackage[bottom]{footmisc}
\usepackage{xcolor}  %the next few lines set the colors of the hyperlinks, and remove the color boxes
\usepackage{geometry}
\usepackage[english,french]{babel}
\usepackage{blindtext}

\usepackage[bitstream-charter]{mathdesign} %In this one, replace \nu with \!\nu
\usepackage[T1]{fontenc}

\hypersetup{
    colorlinks,
    linkcolor={red!50!black},
    citecolor={blue!70!black},
    urlcolor={blue!80!black}
}

\newtheorem{theorem}{Theorem}[section]
\newtheorem*{theorem*}{Theorem}

\newtheorem{proposition}[theorem]{Proposition}

\newtheorem{lemma}[theorem]{Lemma}

\theoremstyle{definition}    
\newtheorem{definition}[theorem]{Definition}

\newtheorem{example}[theorem]{Example}

\newcommand{\pair}[2]{\langle #1, #2 \rangle}
\newcommand{\ignore}[1]{}

\newcommand{\ol}[1]{\overline{#1}}

\newcommand{\ti}[1]{\widetilde{#1}}
\newcommand{\wh}[1]{\widehat{#1}}

\newcommand{\tn}[1]{\textnormal{#1}}
\renewcommand{\i}{{\mathrm{i}}}
\newcommand{\Zv}{Z_{\!\bm{\nu}}}

\def\Ad{\ensuremath{\textnormal{Ad}}}

\def\g{\ensuremath{\mathfrak{g}}}

\def\R{\ensuremath{\mathcal{R}}}
\def\S{\ensuremath{\mathfrak{S}}}
\def\f{\ensuremath{\mathfrak{f}}}

\def\Cl{\ensuremath{\textnormal{Cl}}}
\def\Cliff{\ensuremath{\textnormal{Cliff}}}

\def\bC{\ensuremath{\mathbb{C}}}

\def\bZ{\ensuremath{\mathbb{Z}}}

\def\End{\ensuremath{\textnormal{End}}}

\def\id{\ensuremath{\textnormal{id}}}
\def\pt{\ensuremath{\textnormal{pt}}}

\def\index{\ensuremath{\textnormal{Index}}}
\def\supp{\ensuremath{\textnormal{supp}}}

\def\KK{\ensuremath{\textnormal{KK}}}

\def\K{\ensuremath{\textnormal{K}}}

\begin{document}
\newgeometry{bottom=2.5cm, top=2cm, hmargin=2.5cm}
\sloppy
\title{A KK-theoretic perspective on quantization commutes with reduction}

\author{Rudy Rodsphon}
\date{July 31, 2025}
\address{Sichuan University - School of Mathematics \\ 
No. 24 South Section 1, First Loop Road \\ Chengdu, P. R. China, 610065}
\email{rudy.rodsphon@scu.edu.cn ; rrudy.math@icloud.com}

\maketitle

\vspace{-0.5cm}
\setlength{\parindent}{0mm}

%\begin{abstract}
%Nous revisitons les travaux de Paradan et Vergne sur le probl\`{e}me de la quantification commutant \`{a} la r\'{e}duction sous l'angle de la KK-th\'{e}orie et d'un formalisme r\'{e}cent introduits par Kasparov, nous concentrant plus particuli\`{e}ment sur les parties concernant la th\'{e}orie de l'indice, notamment leur "formule de localisation non-abelienne \`{a} la Witten". Bien que faisant intervenir des ingr\'{e}dients similaires, nos m\'{e}thodes font aussi appara\^{i}tre des simplifications conceptuelles int\'{e}ressantes, et mettent clairement en \'{e}vidence un lien avec les %approches de Ma--Tian--Zhang.   \end{abstract}

\selectlanguage{english}

\begin{abstract}
We reframe Paradan--Vergne's approach to quantization commutes with reduction in KK-theory through a recent formalism introduced by Kasparov, focusing more especially the index theoretic parts that lead to their "Witten non-abelian localization formula". While our method uses the same ingredients as their's in spirit, interesting conceptual simplifications occur, and the relationship to the Ma--Tian--Zhang analytic approach becomes quite transparent. 
\end{abstract}

\vspace{5mm}

\section*{Introduction} \setcounter{section}{0}
The quantization-commutes-with-reduction problem introduced by Guillemin--Sternberg \cite{GS82} has generated a great deal of interest since its inception, particularly because of the wide range of mathematics that have been involved in its study. Loosely speaking, the quantization-commutes-with-reduction principle establishes relationships between equivariant cohomology and the cohomology of `reduced spaces', which involve quotients and a moment map. The original problem solved by Guillemin and Sternberg concerned the action of compact Lie groups on compact K\"ahler manifolds. Meinrenken \cite{Mein98} later generalized this result to compact symplectic manifolds equipped with a Hamiltonian action of a compact Lie group, using symplectic surgery techniques. More recent approaches use localization techniques that are deeply rooted in index theory, and have the advantage to be applicable in wider situations, e.g involving non-compact symplectic manifolds (also referred to as the Vergne conjecture). They mostly split into two schools of thought:

\begin{enumerate}[leftmargin=8mm]
\item[$\bullet$] Analytic techniques (inspired by Witten's deformation of the de Rham complex) introduced in this context by Tian--Zhang  \cite{TZ98}, subsequently extended by Ma--Zhang \cite{MZ14} in the optic of solving the Vergne conjecture, and systematized by Braverman \cite{Brav02}.

\item[$\bullet$] K-theoretic and transverse index techniques (in the sense of Atiyah \cite{AtiTransEll}) introduced in this context by Paradan \cite{Par01}, subsequently extended by Paradan--Vergne \cite{PV14, PV14Wit} also in view of tackling the Vergne conjecture, in the slightly more general setting of spin$^c$ manifolds. 
\end{enumerate}

Note also the work of Hochs--Song \cite{HS17}, which consists in implementing the analytic approach within the aforementioned framework of Paradan and Vergne. \\

The present article proposes a different viewpoint on the index theoretic constructions and results used by Paradan and Vergne \cite{PV14Wit} within a recent formalism developed by Kasparov \cite{Kas16}. The latter reframes and generalizes extensively Atiyah's work on transversally elliptic operator in view of deriving a powerful index theorem, which relates Atiyah's transverse analytic index to a topological index. To be more precise, we prove that Paradan--Vergne's non-abelian localization theorem, which is a key step in their solution to the quantization commutes with reduction problem, is a direct consequence of Kasparov's index theorem together with natural functorial arguments. The context is as follows. \\

Let $(M,g)$ be a (complete Riemannian) manifold without boundary equipped with the action of a compact Lie group $G$ with Lie algebra $\g$, let $E$ be a $G$-equivariant Clifford module bundle over $M$, let $\!\nu : M \to \g^{*}$ be a moment map (in a weak sense, see Section \ref{sec:moment map}), and assume that $0$ is a regular value of $\nu$. The zero set of the vector field $\!\bm{\nu}$ associated to $\!\nu$ admits a decomposition
\[ \Zv = \bigsqcup_{\beta \in \mathcal{B}} G \cdot \big(M^{\beta} \cap \nu^{-1}(\beta) \big) \]
where $\mathcal{B} \subset \g^*$ is a finite set consisting of representatives of the coadjoint orbits within $\!\nu(Z_{\!\bm{\nu}})$.  The zero set $\Zv$ can be singular, but the idea to overcome this issue is to work, for each of its components in the above decomposition,  in appropriate open neighborhoods that are diffeomorphic to "slices". With this in mind, Paradan--Vergne's theorem calculates the transverse index of the Dirac operator by localization around the zero set $\Zv$, and can be stated as follows:

\begin{theorem*} \label{thm:main}
The transverse index of the symbol $\sigma_{\!\nu}(x,\xi) = \i c(\xi + \!\bm{\nu}) \in C^{\infty}(T^*M, \pi_{T^*M}^*\End(E))$
is given by the fixed point formula:
\[ \mathrm{Index}_G [\sigma_{\!\nu}] = \sum_{(V_\pi,\pi) \in \widehat{G}} \big([E_0 \otimes (\nu^{-1}(0) \times_G V_\pi^*)] \otimes_{C(M_0)}[D_0] \big) \,  V_\pi + 
\underset{\beta \in \mathcal{B}}{\sum_{\beta \neq 0}}  \,\mathrm{Index}_{G_\beta} \dfrac{[ \Lambda^{\bullet}_\bC (\mathfrak{g}/\g_\beta)] \boxtimes [\left.\sigma_{\!\nu}\right|_{Y_\beta}]}{\sum_k (-1)^k [ \Lambda^{k} N_\beta ]} \in \widehat{R}(G) \]
where $E_0$ is the Clifford module bundle over the reduced space $M_0=\!\nu^{-1}(0)/G$ induced by $E$ ; $[E_0 \otimes (\nu^{-1}(0) \times_G V_\pi^*)]$ is the K-theory class of the vector bundle $E_0 \otimes (\nu^{-1}(0) \times_G V_\pi) \to M_0$ ; $G_\beta \subset G$ denotes the stabilizer subgroup of $\beta \in \mathcal{B}$ relative to the coadjoint action of $G$ on $\mathfrak{g}^*$ ; $\g_\beta = \mathrm{Lie}(G_\beta)$ ; $Y_\beta$ is a small open neighborhood of $M^\beta \cap \nu^{-1}(\beta)$ inside the fixed point set $M^\beta$ ; $[\sigma_{\!\nu}] \in K^0_G(T_G^*M)$ and $[\left.\sigma_{\!\nu}\right|_{Y_\beta}] \in K^0_{G_\beta}(T_{G_\beta}Y_\beta)$ denote respectively the symbol classes of $\sigma_{\!\nu}$ and of its restriction to $T^*Y_\beta$; and $N_\beta$ is the normal bundle of $M^{\beta}$ in $M$.
\end{theorem*}

In the above, $K^0(T_{G}^{*} M)$ is a receptacle for symbol classes of transversally elliptic operators introduced by Atiyah, which is defined as the K-theory of the transverse (co)tangent bundle $T_G^*M$. In the paper, it will subsequently be replaced by  Kasparov's $K$-group $K_0(\Cl_\Gamma(TM))$, which offers the possibility to define an intrinsic topological tranvserse index. \\

The present work exclusively focuses on the index theoretic parts of Paradan--Vergne's work, and as such, we use geometric arguments that are due to them, e.g the decomposition of the zero set $\!\nu(Z_{\!\bm{\nu}})$, or the construction of certain complex structures of the normal bundles of fixed point sets. As for the index theoretic part, what we do is not far remote from Paradan--Vergne's in spirit. Nevertheless, it should be noted that our work is not a mere translation of theirs in KK-theory. Instead, Kasparov's formalism enables quite substantial simplifications (perhaps at the cost of using a more sophisticated machinery), which are mostly due to the existence of the KK-product and of a topological (transverse) index. As a consequence, the process ends up being orthogonal to Paradan--Vergne's approach, and provides a direct derivation of the aforementioned formula from a simple systematic use of appropriate Thom isomorphisms. In Paradan--Vergne's work, the unavailability of these powerful functorial tools hampers the construction of these natural Thom isomorphisms, that they perform  via a rather cumbersome procedure originating from Atiyah--Singer's handling of families of operators. \\

In the course of constructing these Thom isomorphisms, it is also worth mentioning that we avoid the use of a quite non-trivial transverse index calculation of Atiyah for toral actions on $\bC^n$. Rather, the latter becomes a direct consequence of the above formula, whereas it is utilized as a crucial intermediary tool in Paradan--Vergne's approach. \\

Because of the flexibility and generality of Kasparov's framework, we anticipate our methods to cover a wider range of quantization-commutes-with-reduction-type problems, for instance, a case involving a class of Riemannian foliations arising from Lie algebroids in \cite{LLSS24}, which may be treated elsewhere. \\

Lastly, let us note that gathering the material of the present article, our previous work \cite{LRS19} and the work of Hochs--Song, one sees that the analytic and topological approaches to the quantization commutes with reduction are the same up to Poincar\'{e} duality (in the sense of KK-theory). In sum, the present paper advocates KK-theory as an ideal framework to synthesize and unify the different array of techniques developed in view of solving the quantization commutes with reduction problem. 

\vspace{2mm}
 
\subsection*{Plan of the article} Section 1 gives a short account of Kasparov's work \cite{Kas16} on transverse index theory. Section 2 begins with an exposition of Paradan--Vergne's framework \cite{PV14Wit} and ends on a KK-theoretical proof of their non-abelian localization formula.

\subsection*{Acknowledgement}  I want to warmly thank Gennadi Kasparov for many stimulating discussions over the years. The present paper would probably have remained a manuscript note without his impulse.

\subsection*{Funding} This work has received the support of the NSF grants DMS-1952551, DMS-1952557, of the Fundamental Research Funds for the Central Universities provided through Sichuan University and of the Tianfu Emei talent funding provided by the province of Sichuan.   

\subsection*{Framework}
We consider throughout the paper a complete Riemannian manifold $(M, g)$ (without boundary) equipped with an (isometric) action of a compact Lie group $G$ (the results of Section 1 also hold if $G$ is  locally compact and acts properly, isometrically). Let $\tau=TM$ and $\g=\mathrm{Lie}(G)$ denote respectively the tangent bundle of $M$ and the Lie algebra of $G$. We will not make much distinctions between $TM$ and the cotangent bundle $T^*M$, and will most of the time identify them tacitly via the metric. Similarly, $\g$ and its dual $\g^*$ are most of the time identified via a $G$-invariant product on $\g$. If $A, B$ are $C^*$-algebras, $\KK_{\bullet}^G(A,B)$, $K_{\bullet}^G(A)$ and $K^{\bullet}_G(B)$ the associated $G$-equivariant KK-theory, K-theory, K-homology groups. If moreover $A, B$ are $C_0(X)$-algberas,  $\R\KK_{\bullet}^G(X; A,B)$ denotes the representable KK-theory group of $A$ and $B$. 

\vspace{2mm}

\section{Kasparov's approach of transverse index theory} \setcounter{section}{1} \label{Section1}

In this section, we review Kasparov's KK-theorerical approach of transverse index theory for Lie group actions \cite{Kas16} and the associated index theorem. 

\subsection{Orbital Clifford algebra}

Let $(M, g)$  be a complete Riemannian manifold (without boundary) equipped with an (isometric) action of a compact Lie group $G$. For every $x \in M$, the derivative of the action of $G$ on $x$ reads:   
\[ \rho_x: v \in \g \longmapsto \left. \dfrac{d}{d\! t}\right\vert_{t=0} \exp(tv) \cdot x \, \in T_x M \]
and yields in an obvious way a smooth bundle map $\rho: \g_M:=M \times \g \to TM$. There is a natural extension of the $G$-action on $M$ to $\g_M$ given by $g \cdot (x,v) = ( g \cdot x, \Ad(g)v )$, for which $\rho$ is $G$-equivariant. This action being proper, one can equip $\g_M$ with a $G$-invariant Riemannian metric on the bundle $\g_M$.
%$(\, \bm{.} \, , \, \bm{.} \,)_{\g_M}$ 

\begin{definition}
The \emph{orbital tangent field} $\Gamma \subset TM$ is the continuous field of subspaces defined as the image $\Gamma = \mathrm{Im}(\rho)$ of $\rho$. 
\end{definition}

Recall that $\Gamma$ being a \emph{continuous field} means that it admits a set of sections $\{x \mapsto v_x \in \Gamma_x \}$ that spans $\Gamma$ fiberwise, with the requirement that $\Vert v_x \Vert$ be continuous in $x$. If we denote $C_c^{\infty}(M, \g_M)$ the set of smooth sections of the trivial bundle $\g_M = \g \times M$, we define the set of smooth compactly supported sections of $\Gamma$ as the set $\rho(C_c^{\infty}(M, \g_M))$, which makes $\Gamma$ into a continuous field.  As sets, the fibers of $\Gamma$ are the tangent spaces to the orbits which may exhibit dimension jumps, but equipping $\Gamma$ with the  continuous field above gives a way to handle this issue. \\

%Up to rescaling the metric $g$ by a smooth positive $G$-invariant function, we can assume that $\Vert \rho \Vert \leq 1$. 
After choosing a $G$-invariant Riemannian metric on $\g_M$, consider the transpose operator
\[ \rho^t \colon TM \rightarrow \g_M. \]
%which also has norm $\leq 1$, so $g(\rho(v), \rho(v)) \leq \Vert v \Vert_{\g_M}$ for every section $v$ of $\g_M$. We 
and define a \emph{smooth} bundle map $\varphi \colon TM \rightarrow TM$ to be the composition $\varphi=\rho\circ \rho^t$. The relevance of the map $\varphi$ lies in the fact that \emph{a vector $\xi \in T_xM$ is orthogonal to $\Gamma_x$ if and only if $\varphi_x(\xi)=0$}. \\

We now define an orbital version of Kasparov's $C^*$-algebra $\Cl_\tau(M)$, consisting of $C_0$-sections of the Clifford algebra bundle $\Cliff(TM)$ over $M$. The initial step is to attach a Clifford algebra $\Cliff(\Gamma_x)$ to each fiber $\Gamma_x$ of the orbital field $\Gamma$, and the associated family of such spaces $\Cliff(\Gamma) = \bigsqcup_{x \in M} \Cliff(\Gamma_x) \subset \Cliff(TM)$ inherits a continuous field structure from $\Gamma$.  

\begin{definition} [Kasparov  \cite{Kas16}] \label{def:CliffGamma}
The \emph{orbital Clifford algebra} $\Cl_\Gamma(M)$ is the $C^\ast$-algebra generated by sections of the continuous field $\Cliff(\Gamma)$ over $M$ vanishing at infinity (equipped with the sup-norm). 
\end{definition}

Up to isomorphism, this definition if independent of the choice of the metric on $M$. 

\vspace{2mm}

\subsection{Transverse index and Dirac elements} Let $A$ be a $G$-equivariant pseudodifferential operator with symbol $\sigma_A$ acting on sections of a $G$-equivariant Hermitian vector bundle $E$ \footnote{in the usual H{\"o}rmander $(\rho,\delta)=(1,0)$ class}.  Recall that the operator $A$ is said to be \emph{transversally elliptic} if $\supp(\sigma_A)\cap T_G^*M$ is compact, where $T_G^\ast M=\{(x,\xi) \in T^*M \, ; \, \langle \xi, \Gamma_x \rangle = 0 \}$, and $\supp(\sigma_A)$ is the support of the symbol $\sigma_A$ (i.e the subset of $T^\ast M$ where $\sigma_A$ fails to be invertible). Whence it defines naturally a K-theory class $[\sigma_A]  \in \K^0_G(T_G^*M)=\K^G_0(C_0(T_G^*M))$. \\

When $M$ is compact, Atiyah proved \cite{AtiTransEll} that the restriction $A_\pi$ ($\pi \in \wh{G}$) of $A$ to each isotypical component is Fredholm, hence $A$ has a well-defined equivariant index 
\[ \index_G(A)=\sum_{\pi \in \wh{G}} \index(A_\pi)\pi \in \widehat{R}(G)=\bZ^{\wh{G}}.\]
taking values in the generalized character ring $\widehat{R}(G)$. \\

When $M$ is not compact, Atiyah defines the index of $A$ by a reduction to the compact case, via an embedding of a relatively compact neighborhood $U$ of $\supp(\sigma_A) \cap T^*_GM$ into a compact manifold, cf. Section \ref{sec:noncompact} for more details. \\

Suppose now that $A$ has order $0$; we can reformulate the above discussion from a K-homological perspective as follows. 

\begin{proposition} \emph{(cf. for example \cite[Proposition 6.4]{Kas16})} \, The pair $(L^2(M,E), A)$ induces a $K$-homology class 
\[[A] \in K^0(G \ltimes C_0(M)) \]
that we call the \emph{transverse index of $A$}. 
\end{proposition}

If $M$ is compact, then one can also view $[A]$ as a class in the K-homology group $K^0(C^*(G))$ by crushing $M$ to a point, and the Peter-Weyl theorem shows that $[A]$ coincides with Atiyah's index. \\ 

Suppose now that $E$ is a $G$-equivariant ($\bZ_2$-graded) Clifford module bundle over $M$, and let $D$ be the associated Dirac-type operator acting on sections of $E$. As an elliptic operator, the Dirac operator induces a canonical $K$-homology class $[D] \in K^0(C_0(M))$ represented by the $K$-cycle $\left(L^2(M,E), F := D (1+D^2)^{-1/2} \right)$, and the same K-cycle also induces a K-homology class $[D] \in K^0(G \ltimes C_0(M))$, which describes the transverse index of $D$. The ingredients introduced in previous subsection offer an alternative choice of transverse index class: 

\begin{definition} \label{thm:FundClass} (cf. for example \cite{Kas16, LRS19})

\begin{enumerate}[leftmargin=7mm]
\item[(i)]  The pair $(L^2(M,E), F)$ determines a $K$-homology class 
\[ [D_{M,\Gamma}] \in K^0(G \ltimes \Cl_\Gamma(M)), \] where the crossed product $G \ltimes \Cl_\Gamma(M)$ acts on $L^2(M,E)$ by multiplication (and convolution in $G$). This class will be referred to as the \emph{transverse Dirac element}. \\

\item[(ii)] When $E$ is the (complexified) exterior algebra bundle $\Lambda_{\bC}^{\bullet} TM := \Lambda^{\bullet} TM \otimes \bC$, with $D$ being the canonical Dirac-type operator associated to the de Rham differential $d$, the class $[D_{M,\Gamma}]$ promotes to a $K$-homology class 
\[[d_{M,\Gamma}] \in  K^0(G \ltimes \Cl_{\tau \oplus \Gamma}(M)), \] 
where $\Cl_{\tau \oplus \Gamma}(M) := \Cl_\tau(M) \otimes_{C_0(M)} \Cl_\Gamma(M)$. This class will be referred to as the \emph{transverse de Rham-Dirac element}.
\end{enumerate}
\end{definition}

The action of $\Cl_{\tau \oplus \Gamma}(M)$ on $L^2(M, \Lambda_{\bC}^{\bullet} TM)$ is given, on real (co)vectors, by
\[ \xi_1 \oplus \xi_2 \longmapsto \mathrm{ext}(\xi_1) + \mathrm{int}(\xi_1) + \i (\mathrm{ext}(\xi_2) - \mathrm{int}(\xi_2)) \]
where $\xi_1, \xi_2$ are respectively sections of $TM$ and $\Gamma$. \\

The fact that one indeed gets K-homology classes this way is not absolutely obvious. Details can be found in \cite{Kas16} for (ii), and \cite{LRS19} for (i). The general idea is that the $G$-action by convolution allows to make up for the unboundedness of the elements in $[D, \Cl_\Gamma(M)]$.  

%Another alternative way to proceed is to observe that $[D_{M, \Gamma}] \in K^0(G \ltimes \Cl_\Gamma(M))$ is the image of $[D] \in K^0(G \ltimes C_0(M))$ under the projection map  $\Cl_{\Gamma}(M) \to C_0(M)$. 

\vspace{2mm}

\subsection{Transversally elliptic symbols and the symbol algebra $\frak{S}_\Gamma(M)$.}\label{sec:symbalg}

The precise relationship between the classes $[\sigma_A] \in\K^0_G(T_G^*M)$ and $[A] \in K^0(G \ltimes C_0(M))$ is the object of a highly non-trivial index theorem of Kasparov that we describe in next subsection. Before that, a replacement of the algebra $C_0(T_G^*M)$ is needed. 

\begin{definition} \, (\cite{Kas16}, Definition-Lemma 6.2) \, \label{def:symbalg}
The \emph{symbol algebra} $\S_\Gamma(M)$ is the norm-closure in $C_b(T^*M)$ (the algebra of continuous bounded functions on $T^*M$) of the set of all smooth, bounded functions $b(x,\xi)$ on $T^*M$, which are compactly supported in the $x$-variable, and satisfy (a) together with either (b) or (c) (which are equivalent assuming (a) is satisfied):
\begin{enumerate}
\item The exterior derivative $d_xb(x,\xi)$ in $x$ is norm-bounded uniformly in $\xi$, and there is an estimate $|d_\xi b(x,\xi)|\le C(1+ \Vert \xi \Vert)^{-1}$ for a constant $C$ which depends only on $b$ but not on $(x,\xi)$.

\item The restriction of $b$ to $T_G^*M$ belongs to $C_0(T_G^*M)$. 

\item For any $\varepsilon >0$ there exists a constant $c_\varepsilon>0$ such that
\[ |b(x,\xi)| \leq c_\varepsilon \frac{1+ \Vert \varphi_x(\xi)\Vert^2}{1+ \Vert \xi \Vert^2}+\varepsilon, \qquad \forall x \in M, \xi \in T_xM.\]
\end{enumerate}
\end{definition}

Loosely speaking, item (c) says that $\frak{S}_\Gamma$ consists of  symbols with negative order in the transverse directions. \\

Let $\pi_{T^*M}: T^*M \to M$ denote the canonical projection. Given a $G$-equivariant $\bZ_2$-graded Hermitian vector bundle $E$, we can similarly define a Hilbert $\S_\Gamma(M)$-module, denoted $\S_\Gamma(E)$, as the norm-closure in the space of bounded sections of the pull-back bundle $\pi_{T^*M}^\ast E$ satisfying similar conditions to those in Definition \ref{def:symbalg} (using the norm on the fibres of $\pi_{T^*M}^\ast E$ induced by the Hermitian structure). \\

From now on, we refer to transversally elliptic operators (or symbols) according to the following definition.

\begin{definition} Let $A$ be a properly supported, odd, self-adjoint $G$-invariant pseudodifferential operator of order 0 acting on sections of a $G$-equivariant $\bZ_2$-graded Hermitian vector bundle $E$. We will say that $A$ (or its symbol $\sigma_A$) is transversally elliptic if for every $a \in C_0(M)$,  $a\cdot(1-\sigma_A^2)\in \S_\Gamma(M)$. 
\end{definition}

Therefore, a transversally elliptic symbol naturally determines a class
\[ [\sigma_A]=[(\S_\Gamma(E),\sigma_A)]\in \R\KK^G(M;C_0(M),\S_\Gamma(M)).\]
By construction there is a $\ast$-homomorphism $\S_\Gamma(M)\rightarrow C_0(T_G^*M)$, hence a map
\[ \R\KK^G(M;C_0(M),\S_\Gamma(M))\rightarrow \R\KK^G(M;C_0(M),C_0(T_G^*M)).\]
In this sense the element $[\sigma_A]\in \R\KK^G(M;C_0(M),\S_\Gamma(M))$ can be viewed as a `refinement' of the `naive' class in $\R\KK^G(M;C_0(M),C_0(T_G^*M))$ defined by the symbol. 

\vspace{2mm}

\subsection{Kasparov's index theorem} To state Kasparov's index theorem relating the classes  $[A] \in K^0(G \ltimes C_0(M))$ and $ [\sigma_A]\in \R\KK^G(M;C_0(M),\S_\Gamma(M))$, it will be convenient to introduce the $C^*$-algebra
\[ \Cl_{\Gamma}(TM) := C_0(TM) \otimes_{C_0(M)} \Cl_{\Gamma}(M). \]
which is KK-equivalent to the symbol algebra $\S(M)$ through the KK-class described in the definition below. 

\begin{definition} \label{def:transell} \cite[p. 1344]{Kas16} \,
The element $[\f_{M,\Gamma}]\in \R\KK^G(M;\S_\Gamma(M),\Cl_\Gamma(TM))$ is the class represented by the pair $(\Cl_\Gamma(TM),\f_{M,\Gamma})$ where at a point $(x,\xi) \in T_xM$, the operator $\f_{M,\Gamma}(x,\xi)$ is left Clifford multiplication by $-\i \varphi_x(\xi)(1+ \Vert \varphi_x(\xi) \Vert^2)^{-1/2}$.
\end{definition}

This definition is in fact one of the main reasons to introduce the symbol algebra $\frak{S}_{\Gamma}(M)$. Being the symbol class of an orbital Dirac element, the class $[\f_\Gamma]$ should be thought of as an orbital Bott element. For convenience, let us define an alternative receptacle for symbol classes.

\begin{definition} Let $A$ be a transversally elliptic operator, and let $[\sigma_A] \in \R\KK(M; C_0(M), \S_{\Gamma}(M))$ be its standard symbol class. The \emph{tangent Clifford symbol class} of $A$ is the element $[\sigma_A^{\mathrm{tcl}}]$ obtained as the KK-product:
\[ [\sigma_A^{\mathrm{tcl}}] := [\sigma_A] \otimes_{\S_{\Gamma}(M)} [\f_{\Gamma}] \in \R\KK(M, C_0(M), \Cl_{\Gamma}(TM)) \]
\end{definition}

Heuristically, the meaning of this definition is that one completes $\sigma_A$ into an elliptic symbol, which is paired with the Dolbeault element to recover the transverse index of $A$ from the classical index theorem. To convert this vague idea into a theorem, we will need to define the appropriate Dolbeault element. \\ 

Notice first that the $C^*$-algebras $\Cl_{\Gamma}(TM)$ and $\Cl_{\tau \oplus \Gamma}(M)$ are KK-equivalent. Indeed, recall that the $C^*$-algebras $C_0(TM)$ are KK-equivalent via an element $[d_\xi]\in \R\KK^G(M;C_0(TM),\Cl_\tau(M))$ referred to as the \emph{fiberwise Dirac element}. It is represented by a family of Dirac operators acting on the fibres $TM$, i.e the pair $\left( (L^2(T_xM) \otimes \Cl_{\tau_x})_{x \in M}, (F_x = D_x (1+D^2_x)^{-1/2})_{x \in M} \right)$, where for each $x \in M$, $D_x$ is a Dirac operator acting on the fiber $T_xM$. Its inverse can be described explicitely via a family of fiberwise Bott elements. The element 
\[ [d_\xi] \otimes_{C_0(M)} 1_{\Cl_\Gamma(M)} \in \R\KK^G(M;\Cl_\Gamma(TM),\Cl_{\tau \oplus \Gamma}(M))) \]
then implements a KK-equivalence between $\Cl_{\Gamma}(TM)$ and $\Cl_{\tau \oplus \Gamma}(M)$. Motivated by the fact that the Dirac element on $TM$ induced by the Dolbeault operator splits as a product between $[d_\xi]$ and the Dirac element induced by the de Rham operator, Kasparov makes the following definiton. 

\begin{definition} \textbf{}
%\begin{enumerate}[leftmargin=7mm] 
%\item[(i)] The \emph{fiberwise transverse Dirac element} is the class $[d_{\xi, \Gamma}]$ obtained via the KK-product:
%\[ [d_{\xi, \Gamma}] := [\f_{\Gamma}] \otimes_{\Cl_{\Gamma}(TM)}([d_\xi] \otimes_{C_0(M)} 1_{\Cl_\Gamma(M)})   \in \R\KK^G(M; \S_{\Gamma}(M),\Cl_{\tau \oplus \Gamma}(M))) \]
%\item[(ii)] The \emph{transverse Dolbeault element} is the class $[d_{\xi, \Gamma}]$ obtained via the KK-product:
%\[ [\overline{\partial}_{M, \Gamma}] := j^G ([d_{\xi, \Gamma}]) \otimes_{G \ltimes Cl_{\tau \oplus \Gamma}(M)} [d_{M, \Gamma}]  \in K_G^0(G \ltimes \S_{\Gamma}(M)) \]
%where $j^G$ is the descent map. \\[-2mm]
The {transverse Dolbeault element} is the class $[\overline{\partial}^{\mathrm{cl}}_{TM,\Gamma}]$ obtained as the KK-product
\[ [\overline{\partial}^{\mathrm{cl}}_{TM,\Gamma}]=j^G([d_\xi]\wh{\otimes}_{M}1_{\Cl_\Gamma(M)})\wh{\otimes}_{G\ltimes \Cl_{\tau\oplus\Gamma}(M)}[d_{M,\Gamma}] \in \KK(G\ltimes \Cl_\Gamma(TM),\bC).
\]%\footnote{Recall that for every $G$-$C^*$-algebras $A$ and $B$, the descent map $j^G: KK^G(A,B) \to KK(G \ltimes A, G \ltimes B)$.}.
\end{definition}

We can now state Kasparov's index theorem.

\begin{theorem} \emph{(Kasparov, \cite[Theorem 8.18]{Kas16})} \label{thm:KaspInd}

Let $(M, g)$ be a complete Riemannian manifold (without boundary) equipped with a proper and isometric action of a Lie group $G$. Let $A$ be a transversally elliptic operator of order $0$ on $M$, with symbol $\sigma_A$. Then
\[ [A]=j^G[\sigma_A^{\mathrm{tcl}}] \otimes_{G\ltimes \Cl_\Gamma(TM)}[\overline{\partial}^{\mathrm{cl}}_{TM,\Gamma}] \in \K^0(G\ltimes C_0(M)),\]
where $j^G$ denotes the descent map. 
\end{theorem}
%\[ [A] = j^G [\sigma_A] \otimes_{\S(M)} [\overline{\partial}_{M, \Gamma}] \in K^0(G \ltimes C_0(M)) \]

It is not hard to show that in the setup of the classical index theorem (i.e $G$ is the trivial group), the right-hand-side is exactly the topological index of Atiyah--Singer \cite{AtiSinI}. In general, we can therefore refer to it as a \emph{topological transverse index}. \\

When $\sigma_A$ is transversally elliptic in Atiyah's sense, note that $[\sigma_A^{\mathrm{tcl}}]$ can also be seen as a class in $K_0(\Cl_{\Gamma}(TM))$; the next section discusses this point further together with the compatibility between Atiyah's and Kasparov's approaches. 

\vspace{2mm}

\subsection{Relationship to Atiyah's index and reduction to compact manifolds}  \label{sec:noncompact}

Atiyah \cite{AtiTransEll} defines the transverse index more generally for any element $\alpha_M \in \K^0_G(T_GM)$ where $M$ is a not-necessarily compact $G$-manifold.  The construction proceeds as follows.  Atiyah proves \cite[Lemma 3.6]{AtiTransEll} that one can find a $\bZ_2$-graded Hermitian vector bundle $E=E_0\oplus E_1$ on $M$ and $\sigma_M \in C_b(TM,\pi_{TM}^\ast \End(E))$ an odd, self-adjoint bundle endomorphism  whose restriction to $T_GM$ represents the class $\alpha$, and such that one has $\sigma_M^2=1$ outside $\pi_{TM}^{-1}(K)$ for a $G$-invariant compact subset $K$ of $M$.  Choose a Hermitian vector bundle $F \rightarrow M$ such that $\ti{E}_0=E_0\oplus F$ is trivial, and fix a trivialization.  Let $\ti{E}_1=E_1\oplus F$ and $\ti{\sigma}_M=\sigma_M\oplus \id_F$.  Via $\ti{\sigma}_M$ we obtain a trivialization of $(E_1\oplus F)|_{M\setminus K}$.  Choose a relatively compact $G$-invariant open neighborhood $U$ of $K$, and let $\iota_{U,M}\colon U \hookrightarrow M$ be the inclusion; we will use the same symbol for the induced open inclusion $T_GU\hookrightarrow T_GM$.  The pair $(\ti{E}|_U,\ti{\sigma}_M|_U)$ represents a class $\alpha_U\in \K^0_G(T_GU)$ and $\alpha_M=(\iota_{U,M})_\ast \alpha_U$ by construction.  Choose a $G$-equivariant open embedding $\iota_{U,X}$ of $U$ into a compact $G$-manifold $X$; again we use the same symbol for the induced open inclusion $T_GU\hookrightarrow T_GX$.  Using the trivializations over $U\setminus K$, the bundle $\ti{E}|_U$ and endomorphism $\ti{\sigma}_M|_U$ can be extended trivially to $X$ (denoted $\ti{E}_X$, $\ti{\sigma}_X$ respectively) and represent the class $\alpha_X=(\iota_{U,X})_\ast \alpha_U \in \K^0_G(T_GX)$.  Atiyah defines
\[ \index_G(\alpha_M)=\index_G(A_X) \in \widehat{R}(G) \]
where $A_X$ is any transversally elliptic operator on $X$ such that the (naive) K-theory class of its symbol is $\alpha_X$.  Atiyah proves an excision property  \cite[Theorem 3.7]{AtiTransEll} showing that the index can be determined just from data on $U$, and hence the construction is independent of the various choices.\\

In Kasparov's framework, this construction can be reformulated as follows. Suppose that one manages to choose $\sigma_M$ such that, in addition to the conditions above, one has $(1-\sigma_M^2)\in \S_\Gamma(M)$.  Then $\sigma_M$ determines a class $[\sigma_{M,\tn{c}}]=[(\S_\Gamma(E),\sigma_M)]\in \KK^G(\bC,\S_\Gamma(M))$ refining the class $\alpha_M$.  The subscript `c' is to emphasize that this is a K-theory class whose support is compact over $M$, in contrast with the symbols defining elements of the group $\R\KK^G(M;C_0(M),\S_\Gamma(M))$ that were considered in Section \ref{sec:symbalg}.  One then obtains similar classes $[\ti{\sigma}_{U,\tn{c}}]=[(\S_\Gamma(\ti{E}|_U),\ti{\sigma}|_U)]\in \KK^G(\bC,\S_\Gamma(U))$ refining $\alpha_U$, $[\ti{\sigma}_{X,\tn{c}}]\in [(\S_\Gamma(\ti{E}_X),\ti{\sigma}_X)]\in \KK^G(\bC,\S_\Gamma(X))$ refining $\alpha_X$, and moreover
\begin{equation} 
\label{eqn:functorinclusion}
[\ti{\sigma}_{X,\tn{c}}]=(\iota_{U,X})_\ast[\ti{\sigma}_{U,\tn{c}}], \qquad [\sigma_{M,\tn{c}}]=(\iota_{U,M})_\ast[\ti{\sigma}_{U,\tn{c}}].
\end{equation}
Let $[\ti{\sigma}_{X,\tn{c}}^{\tn{tcl}}]$, $[\ti{\sigma}_{U,\tn{c}}^{\tn{tcl}}]$, $[\sigma_{M,\tn{c}}^{\tn{tcl}}]$ be the corresponding tangential Clifford symbols obtained by KK-product with $\f_{X,\Gamma}$, $\f_{U,\Gamma}$, $\f_{M,\Gamma}$ respectively.  Functoriality of the classes $\f_{-,\Gamma}$ under open embeddings implies the tangential Clifford symbols satisfy analogous formulas. \\

Let $p \colon X \rightarrow \pt$ be the collapse map, and $[\sigma_{A,X}]\in \R\KK^G(X;C(X),\S_\Gamma(X))$ the class defined by the symbol of $A_X$, so that $p_\ast[\sigma_{A,X}]=[\ti{\sigma}_{X,\tn{c}}]$.  By Theorem \ref{thm:KaspInd},
\[ \index(A_X)=p_\ast[A_X]=j^G([\ti{\sigma}_{X,\tn{c}}^{\tn{tcl}}])\wh{\otimes}_{G\ltimes \Cl_\Gamma(TX)}[\ol{\partial}^{\tn{cl}}_{TX,\Gamma}].\]
Equation \eqref{eqn:functorinclusion}, as well as the functoriality of the KK-product and of the transverse Dolbeault class, give the equivalent formula
\[ \index(A_X)=j^G([\ti{\sigma}_{U,\tn{c}}^{\tn{tcl}}])\wh{\otimes}_{G\ltimes \Cl_\Gamma(TU)}[\ol{\partial}_{TU,\Gamma}]=j^G([\sigma_{M,\tn{c}}^{\tn{tcl}}])\wh{\otimes}_{G\ltimes \Cl_\Gamma(TM)}[\ol{\partial}^{\tn{cl}}_{TM,\Gamma}].\]
We thus obtain the following formula for the index in Atiyah's sense of $\alpha_M=\iota_{T_GM}^\ast[\sigma_{M,\tn{c}}]\in \K^0_G(T_GM)$:
\begin{equation*}
\label{eqn:indexopenmfld}
\index(\iota_{T_GM}^\ast[\sigma_{M,\tn{c}}])=j^G([\sigma_{M,\tn{c}}^{\tn{tcl}}])\wh{\otimes}_{G\ltimes \Cl_\Gamma(TM)}[\ol{\partial}^{\tn{cl}}_{TM,\Gamma}].
\end{equation*}

\vspace{2mm}

\section{Localization}
We begin with a summary of Paradan--Vergne's setup, following closely their work \cite{PV14Wit} apart from its index theoretic content, which is treated via Kasparov's formalism and KK-theory. The section ends with a KK-theoretical proof of their non-abelian localization formula and a discussion explaining why this point of view unifies different approaches to the quantization commutes with reduction problem. 

\subsection{Moment map and excision} \label{sec:moment map} As in previous section, we work with a complete Riemannian manifold $(M, g)$  equipped with an action of a compact  Lie group $G$. Moreover, we shall identify $\g$ and its dual $\g^*$ via a $G$-invariant product on $\g$ when appropriate. \\ 

Consider in addition that we have a smooth $G$-equivariant map
\[ \nu : M \to \g \simeq \g^{*} \]
such that the zero set $\Zv$ of the (orbital) vector field:
\[ \!\bm{\nu}(x) = \rho_x(\!\nu(x)) =  \left. \dfrac{d}{dt}\right\vert_{t=0} \exp(t \!\nu(x)) \cdot x \] 
is \emph{compact}. As is customary when dealing with the quantization commutes with reduction problem, we suppose that $0$ is a regular value of $\!\nu$. \\

We now use the vector field $\!\bm{\nu}$ to perturb the symbol of the Dirac operator.  

\begin{definition}
Let $\sigma_{\!\nu} \in C^\infty(T^*M, \pi_{T^*M}^* \End(E))$ be the symbol defined by
\[ \sigma_{\!\nu}(x, \xi) = \i c(\xi + \!\bm{\nu}) \, ; \, \forall (x,\xi) \in T^*M \]
\end{definition}

This is a transversally elliptic symbol, as the set on which $\sigma_{\!\nu}$ fails to be invertible is precisely the compact set $\Zv$. Choose a relatively compact neighborhood $U$ of $\Zv$. Without loss of generality, we can suppose that $|\!\bm\nu|=1$ outside $U$. Then, the vector field $\!\bm{\nu}$ induces a K-theory class
\[  [\!\nu] := [(\Cl_{\Gamma}(M), c(\!\bm{\nu})] \in K_0(\Cl_{\Gamma}(M)) \]
In \cite{LRS19}, it is proved that the transverse index 
\[ \mathrm{Index}_G [\sigma_{\!\nu}] \in \widehat{R}(G) \]  
of $[\sigma_{\!\nu}] \in K^0(\Cl_{\Gamma}(TM))$ is provided by the KK-product:
\[ \mathrm{Index}_G [\sigma_{\!\nu}]  = j^G [\!\nu]  \otimes_{\Cl_{G \ltimes \Gamma}(M)} [D_{M,\Gamma}] = j^G[\sigma_{\!\nu}] \otimes_{G\ltimes \Cl_\Gamma(TM)}[\overline{\partial}^{\mathrm{cl}}_{TM,\Gamma}] \in \widehat{R}(G)\]
Denoting $\iota_{U,M}: \Cl_\Gamma(U) \to \Cl_\Gamma(M)$ the natural extension-by-0 homomorphism, one sees without difficulty that $(\iota_{U,M})_{*}[\!\nu|_{U}] = [\!\nu]$, where $[\!\nu|_{U}] \in K^0(\Cl_\Gamma(U))$ is the class defined similarly to $[\!\nu]$ from the restriction $\!\nu|_U$ of $\!\nu$ to $U$.  Associativity of the KK-product together with the formula above applied to the manifold $U$ therefore yield the following excision/localization result:

\begin{proposition} \label{prop: localization 1} Let $U$ be a relatively compact open neighborhood of the zero set $\Zv$. Then, 
\begin{equation} \label{eqn: localization 1}
\mathrm{Index}_G [\sigma_{\!\nu}] = [\!\nu|_U]  \otimes_{\Cl_{G \ltimes \Gamma}(U)} [D_{U,\Gamma}] = j^G[\left.\sigma_{\!\nu}^\mathrm{tcl}\right|_U] \otimes_{G\ltimes \Cl_\Gamma(TU)}[\overline{\partial}^{\mathrm{cl}}_{TU,\Gamma}] \in \widehat{R}(G) 
\end{equation}
\end{proposition} 

(The rightmost side is Kasparov's formula applied to the manifold $U$). The middle part of the formula is set merely in view of the final discussion of the paper, but is not necessary in the argument above: the relationship between the leftmost and rightmost sides can in fact be derived directly as in Section \ref{sec:noncompact}.
\vspace{2mm} 

\subsection{Coadjoint orbits and slices}
Next, we describe, following Paradan--Vergne,  a neighborhood of the zero set $\Zv$ in which the Equation (\ref{eqn: localization 1}) becomes a fixed-point formula.  For that purpose, we need an additional piece of structure on the smooth map $\!\nu$. It is a weakening of the notion of Hamiltonian action which requires neither the 2-form of $M$ to be symplectic, nor non-degenerate. 

\begin{definition}  \label{def:moment map}
We say that the smooth $G$-equivariant map $\!\nu$ is a \emph{moment map} if there exists a $G$-invariant closed 2-form $\omega$ on $M$ such that 
\[ \mathrm{int}(\bm{\beta}) \omega + d \langle \!\nu, \beta \rangle = 0 \, , \, \forall \beta \in \g \]
where $\bm{\beta}(x) = \left. \frac{d}{d \!t}\right|_{t=0} e^{-t\beta} \cdot x$ is the vector field on $M$ generated by $\beta \in \g$, and $\mathrm{int}(\bm{\beta})\omega$ denotes the contraction of the differential form $\omega$ by the vector field $\bm{\beta}$.
\end{definition}

Here is an example of moment map to have in mind, because of its relevance to the quantization commutes with reduction problem.  

\begin{example} Let $L$ be a $G$-equivariant line bundle endowed with a connection $\nabla$, and let $\omega = -\i R$, where $R$ is the curvature 2-form of $\nabla$. The \emph{Lie derivative} $\mathcal{L}_{\!\nu}$ acting on smooth sections of $L$ is defined by
\[ (\mathcal{L}_{\beta} \varphi)(x)=\left. \dfrac{d}{dt} \right\vert_{t=0} e^{t\beta} \cdot \phi(e^{-t \beta} x) .\]
Then, the smooth map $\nu :M  \to \g^\ast$ defined by the equation
\[ \pair{\nu}{\beta} = \nabla_{\bm{\beta}}-\mathcal{L}_{\beta} \, , \, \, \, \, \forall \beta \in \g \]
is a moment map.
\end{example}

Under the additional assumption that $\!\nu$ is a moment map, one observes that the compact zero set $\nu(\Zv) \subset \g^*$ is a finite union of coadjoint orbits. Then, let $\mathcal{B}$ be a finite set consisting of representatives of those coadjoint orbits covering $\nu(\Zv)$. This leads to the following description of $\Zv$:  
\[ \Zv = \bigsqcup_{\beta \in \mathcal{B}} G \cdot \big(M^{\beta} \cap \nu^{-1}(\beta) \big) \]
where $M^{\beta} = \left\{ x \in M \, ; \, \bm{\beta}(x) = 0  \right\}$ is the stabilizer of the infinitesimal action of $\beta \in \mathcal{B} \subset \g^{*}$. (Recall that $\bm{\beta}$ is the vector field $\bm{\beta}(x) = \left. \frac{d}{d \!t}\right|_{t=0} e^{-t\beta} \cdot x$). \\

We now describe, for each component $G \cdot \big(M^{\beta} \cap \nu^{-1}(\beta) \big)$ in the decomposition above, a neighborhood diffeomorphic to a `slice'. In what follows, let $\g_{\beta} = \mathrm{Lie}(G_\beta)$, where $G_\beta \subset G$ is the stabilizer subgroup of $\beta$ with respect to the coadjoint action.   

\begin{proposition} \cite[Proposition 8.4]{PV14Wit}
Let $\beta \in \mathcal{B} \smallsetminus \{0\}$. Then, for any sufficiently small neighborhood $W_\beta \subset \g^{*}_\beta$ of $\beta$ in $\g^{*}_\beta$, 
\begin{enumerate}[leftmargin=7mm]
\item[\emph{(i)}] $V_\beta := \nu^{-1}(W_\beta)$ is a $G_\beta$-invariant submanifold of $M$. 
\item[\emph{(ii)}] $G \cdot V_\beta = \nu^{-1}(G \cdot W_\beta)$ is diffeomorphic to the slice $U_\beta := G \times_{G_\beta} V_\beta$.  
\end{enumerate}
Moreover, if $\beta=0$ is a regular value of $\nu$, then $G$ acts locally freely on the submanifold $\nu^{-1}(0)$. 
\end{proposition}

Whence the \emph{reduced space} $M_0 := \nu^{-1}(0)/G$ is an orbifold, and inherits a natural $G$-equivariant Clifford module structure $E_0$. To keep the exposition simple, let us suppose that the above action is \emph{free}, so that $M_0$ is a manifold. In the noncommutative geometric viewpoint  adopted, this is a very minor technical point and does not make much difference. \\

At this point, Equation (\ref{eqn: localization 1}) of Proposition \ref{prop: localization 1} reads
\begin{equation} \label{eqn: localization 2}
\mathrm{Index}_G [\sigma_{\!\nu}] = \mathrm{Index}_G[\sigma_{\!\nu}|_{U_0}]+ \underset{\beta \in \mathcal{B}}{\sum_{\beta \neq 0}} j^G[\left.\sigma_{\!\nu}^\mathrm{tcl}\right|_{U_\beta}] \otimes_{G\ltimes \Cl_\Gamma(TU_\beta)}[\overline{\partial}^{\mathrm{cl}}_{TU_\beta,\Gamma}] \in \widehat{R}(G) 
\end{equation}  
where $U_0$ is a tubular neighborhood of $\nu^{-1}(0)$ that we choose small enough to be acted on freely by $G$. \\

The first term $\mathrm{Index}_G[\sigma_{\!\nu}|_{U_0}] := j^G[\left.\sigma_{\!\nu}^\mathrm{tcl}\right|_{U_0}] \otimes_{G\ltimes \Cl_\Gamma(TU_0)}[\overline{\partial}^{\mathrm{cl}}_{TU_0,\Gamma}] \in \widehat{R}(G)$ represents the transverse index of the symbol $\sigma_{\!\nu}|_{U_0}$, but since $G$ acts freely on $U_0$, the algebras $G \ltimes C_0(U_0)$ and $C_0(U_0/G)$ are Morita equivalent, so $\sigma_{\!\nu}|_{U_0}$ can be seen as an elliptic symbol on the manifold $U_0/G$. Identifying $U_0/G$ to the normal bundle of $M_0:=\!\nu^{-1}(0)/G$ standard techniques prove that up to the associated direct image morphism\footnote{or simply Bott periodicity in this case, because with $U_0$ may further be identified to $\nu^{-1}(0) \times \mathfrak{g}^*$}, $\mathrm{Index}_G[\sigma_{\!\nu}|_{U_0}]$ is equal to the equivariant index of the Dirac operator $D_0$ on the reduced space $M_0=\nu^{-1}(0)/G$. In other words:

\begin{proposition} Let $[D_0] \in K_0(M_0)$ denote the $K$-homology class of the Dirac operator on the reduced space $M_0$. Denote $E_0 = E|_{M_0}$ the Clifford module bundle on $M_0$ induced by $E$. Then,
\[ \mathrm{Index}_G[\sigma_{\!\nu}|_{U_0}] = \sum_{\pi \in \widehat{G}} \big([E_0 \otimes (\nu^{-1}(0) \times_G V_\pi^*)] \otimes_{C(M_0)}[D_0] \big) \,  V_\pi  \]
where $[E_0 \otimes (\nu^{-1}(0) \times_G V_\pi^*)] \in K^0(M_0)$ is the K-theory class of the bundle $E_0 \otimes (\nu^{-1}(0) \times_G V_\pi)$ over $M_0$, and $(\bm{.} \, \, \,\otimes_{C(M_0)} \, \bm{.})$ is the KK-product over $C(M_0)$ \emph{(}which is in this situation the index of $D_0$ twisted by the vector bundle in question\emph{)}.
\end{proposition}

In particular, its $G$-invariant part is the term associated to the trivial representation. In substance, the argument given above is quite similar to Paradan-Vergne's; see \cite[Subsection 8.2]{PV14Wit}, which also provides additional details on this well-known procedure (see also \cite{BV96} for an older reference). 

\vspace{2mm}

\subsection{Normal bundles and complex structures} \label{sec: normal}

To deal with the terms of (\ref{eqn: localization 2}) depending on non-zero $\beta \in \mathcal{B} \subset \g^*$, we shall descend to the fixed point sets $M^{\beta}$ by a slightly more elaborated Thom isomorphism. We will now describe complex structures on the normal bundles $N_\beta \to M_\beta$ in $U_\beta = G \times_{G_\beta} V_\beta$ allowing us to do so. Note that by the tubular neighborhood theorem, we can identify $U_\beta$ to the total space of the normal bundle $N_\beta$. \\

Consider a non-zero element $\beta \in \mathcal{B} \subset \mathfrak{g}^*$. Because of the decomposition 
\[ TM|_{M^\beta} = TM^\beta \oplus N_\beta, \]
the infinitesimal action $\mathcal{L}_{\beta}$ on vector fields on $M$ induces a fiberwise linear isomorphism $\mathcal{L}_{\beta} \in \End(N_\beta)$. It is skew-symmetric (with respect to the metric $g$), hence diagonalizable with "imaginary positive" eigenvalues. We can therefore make the following definition:
%\[ \mathcal{L}_{\beta}(v)=\left. \dfrac{d}{dt} \right\vert_{t=0} e^{t\beta} \cdot v .\]

\begin{definition} \label{def:complex structure 1}
We define the complex structure $J_\beta$ on $N_\beta$ as the operator $J_\beta = \mathcal{L}_{\beta} |\mathcal{L}_{\beta}|^{-1}$, where $|\mathcal{L}_{\beta}|^{-1} = (-\mathcal{L}_{\beta}^2)^{1/2}$.     
\end{definition}

A similar construction applies to $Y_\beta := (V_\beta)^{\beta} = M^\beta \cap V_\beta$ in $M^\beta$ (which is open in $M^\beta$), so the normal bundle $\eta_\beta$ relative to the embedding $\iota_{Y_\beta, V_\beta} : Y_\beta \hookrightarrow V_\beta$ comes equipped with a complex structure. \\ 

We will also need another construction of the same type the quotient space $\g/\g_{\beta}$. Recall that $\g$ and $\g^*$ are identified with a $G$-invariant inner product. Via this identification, the endomorphism of $\g$ induced by the adjoint action of $\beta$ yields a linear skew-symmetric isomorphism on $\g/\g_{\beta}$. Hence we can make the following definition.

\begin{definition} \label{def:complex structure 2}
The linear map $\mathfrak{J}_\beta = \beta |\beta|^{-1}$ provides a complex structure on the vector space $\g/\g_{\beta}$. 
\end{definition}

\vspace{2mm}

\subsection{KK-theory and the non-abelian localization formula}
We are now ready to re-interpret Paradan--Vergne's calculation of $\mathrm{Index}_G [\sigma_{\!\nu}]$ in KK-theory. We shall decompose the terms 
\[ j^G[\left.\sigma_{\!\nu}^\mathrm{tcl}\right|_{U_\beta}] \otimes_{G\ltimes \Cl_\Gamma(TU_\beta)}[\overline{\partial}^{\mathrm{cl}}_{TU_\beta,\Gamma}] \]
with non-zero elements $\beta \in \mathcal{B} \subset g^{*}$ in Equation (\ref{eqn: localization 2}), appealing to  appropriate Thom isomorphisms. \\

Recall that $U_\beta := G \cdot V_\beta$ is the slice diffeomorphic to $G \times_{G_\beta} V_\beta$, with $V_\beta$ being the $G_\beta$-invariant submanifold $\nu^{-1}(W_\beta)$ in $M$ obtained from a small neighborhood $W_\beta$ of $\beta$ in the infinitesimal stabilizer $\g_{\beta}$. In addition, consider the fixed point sets $M^\beta = \{x \in M \, ; \, \bm{\beta}(x)=0 \}$ and $Y_\beta = M^\beta \cap V_\beta$, which is a open neighborhood of $M^\beta \cap \nu^{-1}(\beta)$ in $M^\beta$.

\begin{proposition}
For a non-zero $\beta \in \mathcal{B} \subset \mathfrak{g}^*$, let $G_\beta \subset G$ be the stabilizer subgroup of $\beta$ for the coadjoint action of $G$ on $\mathfrak{g}^*$. Then, $j^G[\left.\sigma_{\!\nu}^\mathrm{tcl}\right|_{U_\beta}] \otimes_{G\ltimes \Cl_\Gamma(TU_\beta)}[\overline{\partial}^{\mathrm{cl}}_{TU_\beta,\Gamma}]$ is the transverse index of the symbol
\[ \mathrm{Ind}^G_{G_\beta} \dfrac{[ \Lambda^{\bullet}_\bC (\mathfrak{g}/\g_\beta)] \boxtimes [\left.\sigma_{\!\nu}^\mathrm{tcl}\right|_{Y_\beta}]}{\sum_k (-1)^k [ \Lambda^{k} N_\beta ]}  \in K_0\big(\Cl_{\Gamma}(T(G \times_{G_\beta} Y_\beta))\big) = K_0(\Cl_{\Gamma}(TU_\beta)),\]
where $\boxtimes$ denotes the external product ; $[\left.\sigma_{\!\nu}^\mathrm{tcl}\right|_{Y_\beta}] \in K_0(\Cl_{\Gamma}(TY_\beta))$ is the symbol class associated to the restriction of $\sigma_{\!\nu}$ to $Y_\beta$ ; $\mathrm{Ind}^G_{G_\beta}$ denotes Atiyah's induction functor, and $N_\beta$ is the normal bundle of $M^{\beta}$ in $M$.
\end{proposition} 

\begin{proof}
Let $\Gamma_\beta$ denote the orbital field in $V_\beta$ generated by the action of $G_\beta$. By Bott periodicity, we have an isomorphism of K-theory groups:
\begin{multline*}
K_0^{G_\beta}(\Cl_{\Gamma_\beta}(TV_{\beta})) \longrightarrow K_0^{G_\beta} \big(\Cl_{\Gamma_\beta}(TV_{\beta}) \otimes C_0(\g/\g_\beta) \otimes \mathrm{Cliff}(\g/\g_\beta)\big) \\ = K_0^{G_\beta} \big( C_0(TV_\beta \times \g/\g_\beta) \otimes_{C_0(V_\beta)} \Cl_{\Gamma_\beta \oplus \g/\g_\beta}(V_{\beta})\big)
\end{multline*}
provided as usual by external multiplication with $[\Lambda^{\bullet} (\g/\g_\beta)] \in K_0^{G_\beta}\big(C_0(\g/\g_\beta) \otimes \mathrm{Cliff}(\g/\g_\beta)\big)$, where $\Lambda^{\bullet} (\g/\g_\beta)$ is the exterior algebra of the quotient space $\g/\g_\beta$, which has a natural complex structure (cf. Definition \ref{def:complex structure 2}). \\

Next, consider Kasparov's induction functor $\mathrm{Ind}^{G}_{G_\beta}$: 
\[ K_0^{G_\beta} \big( C_0(TV_\beta \times \g/\g_\beta) \otimes_{C_0(V_\beta)} \Cl_{\Gamma_\beta \oplus \g/\g_\beta}(V_{\beta})\big) \longrightarrow KK_0^{G} \big(C(G/G_\beta),  \Cl_{\Gamma}(TU_\beta) \big) \]

%where we obtain the last KK-group by noticing the equality
%\[ C_0(G \times_{G_\beta} (TV_\beta \times \g/\g_\beta)) \otimes_{C_0(G \times_{G_\beta} V_\beta)} \Cl_{\Gamma}(G \times_{G_\beta}V_{\beta}) = \Cl_{\Gamma}(TU_\beta) \] 
where the second entry in the right-hand side is obtained from the fact that the tangent bundle of the slice $U_\beta$ is given by $T(G \times_{G_\beta} V_\beta) = G \times_{G_\beta} (TV_\beta \times \g/\g_\beta)$. Then, observe that $G/G_\beta = T_G(G/G_\beta)$, and recall that the zero operator on $G/G_\beta$ is $G$-transversally elliptic. The composition of the two homomorphisms above with the (right) KK-product by $[\bm{0}] \in K^0(T_G^{*}(G/G_\beta))$ (that we shall absorb in the induction morphism to avoid heavier notations) therefore yields a homomorphism:
\[ K_0^{G_\beta}(\Cl_{\Gamma_\beta}(TV_{\beta})) \longrightarrow K_0^{G}(\Cl_{\Gamma}(TU_{\beta})) \]
This construction was already considered by Atiyah,  \cite[Section 4]{AtiTransEll}, who also proves that it is an isomorphism (this can be seen directly via the construction and properties of Kasparov's induction functor). Following Atiyah, we still denote this homomorphism $\mathrm{Ind}^{G}_{G_\beta}$ (this is the one considered in the statement of the proposition).  \\

On the other hand, $[\left. \sigma_{\!\nu}^{\mathrm{tcl}} \right|_{U_\beta}] \in K_0^G(\Cl_{\Gamma}(TU_\beta))$ is the image of $[\left. \sigma_{\!\nu}^{\mathrm{tcl}} \right|_{V_\beta}] \in K_0^{G_\beta}(\Cl_{\Gamma}(TV_\beta))$ by this isomorphism, i.e.
\[ [\left. \sigma_{\!\nu}^{\mathrm{tcl}} \right|_{U_\beta}] = \mathrm{Ind}^{G}_{G_\beta} \big([\Lambda^{\bullet} (\g/\g_{\beta})] \boxtimes [\left. \sigma_{\!\nu}^{\mathrm{tcl}} \right|_{V_\beta}] \big)\]  
In addition, the symbol class $[\left. \sigma_{\!\nu}^{\mathrm{tcl}} \right|_{V_\beta}] \in K_0^{G_\beta}(\Cl_{\Gamma}(TV_\beta))$ can be decomposed further by descending to the fixed point set $M^\beta$. To this end, let 
\[ Y_\beta := (V_\beta)^{\beta} = \{ x \in V_\beta \, ; \, \bm{\beta}(x)=0 \} = M^\beta \cap V_\beta \]
which is open in $M^\beta$. Let $\eta_\beta$ be the normal bundle relative to the embedding $\iota_{Y_\beta, V_\beta} : Y_\beta \hookrightarrow V_\beta$ and identify $V_\beta$ with its total space. We endow $\eta_\beta$ with the complex structure described in Section \ref{sec: normal}.

\begin{lemma} \label{lem: KKeq 1} One has a KK-equivalence $\Cl_{\Gamma_\beta}(TY_{\beta}) \sim \Cl_{\Gamma_\beta}(TV_{\beta})$ provided by a "Thom" element:
\[ [\mathcal{T}_\beta] := [T_{\beta}] \otimes 1_{\Cl_{\Gamma_\beta}(Y_\beta)} \in \KK^{G_\beta}\big(\Cl_{\Gamma_\beta}(TY_{\beta}), \Cl_{\Gamma_\beta}(TV_{\beta})\big) \]
where $[T_{\beta}] \in \KK^{G_\beta}\big(C_0(TY_\beta), C_0(TV_\beta)\big)$ is the standard Thom element that implements the KK-equivalence $C_0(TY_\beta) \sim  C_0(TV_\beta)$. %constructed by an appropriate family of Bott elements acting fiberwise on the complex vector bundle $TV_\beta \to TY_\beta. Its inverse is built from a family of  Dirac operators acting fiberwise on the same bundle. Hence, we can write:
\end{lemma}

\begin{proof} (of Lemma \ref{lem: KKeq 1})
%First, recall that by definition, we have 
%\[ \Cl_{\Gamma_\beta}(TV^{\beta}) = C_0(TV_\beta) \otimes_{C_0(V_\beta)} \Cl_{\Gamma_\beta}(V_\beta) \]
This is simply a consequence of the local decomposition 
\[ \Cl_{\Gamma_\beta}(V_\beta) = C_0(V_\beta) \otimes_{C_0(Y_\beta)} \Cl_{\Gamma_\beta}(Y_\beta) \]
%Since $V_\beta$ is identified to the total space of the complex vector bundle $N_\beta$ over $Y_\beta$, we have a KK-equivalence $C_0(V_\beta) \sim C_0(Y_\beta)$, whence a KK-equivalence $\Cl_{\Gamma_\beta}(V_\beta) \sim \Cl_{\Gamma_\beta}(Y_\beta)$. 
combined with the usual Thom isomorphism $C_0(V_\beta) \sim C_0(Y_\beta)$. % (recall $V_\beta$ identifies with the total space $\mathtt{N_\beta} \to Y_\beta$.  
\end{proof}

Tensoring both sides of the KK-equivalence in the lemma by $C_0(\g/\g_\beta) \otimes \mathrm{Cliff}(\g/\g_\beta)$ and combining with the aformentioned Bott periodicity isomorphism yields another a KK-equivalence
\[  C_0(TY_\beta \times \g/\g_\beta) \otimes_{C_0(Y_\beta)} \Cl_{\Gamma_\beta \oplus \g/\g_\beta}(Y_{\beta}) \sim C_0(TV_\beta \times \g/\g_\beta) \otimes_{C_0(V_\beta)} \Cl_{\Gamma_\beta \oplus \g/\g_\beta}(V_{\beta}).\] 
implemented by a Thom element 
\[ [\bm{\mathcal{T}}_{\!\!\beta}] \in \KK^{G_\beta}\big(C_0(TY_\beta \times \g/\g_\beta) \otimes_{C_0(Y_\beta)} \Cl_{\Gamma_\beta \oplus \g/\g_\beta}(Y_{\beta}), C_0(TV_\beta \times \g/\g_\beta) \otimes_{C_0(V_\beta)} \Cl_{\Gamma_\beta \oplus \g/\g_\beta}(V_{\beta}) \big). \]
Since Kasparov's induction functor preserves KK-equivalences, we have
\[ \Cl_{\Gamma}(T(G \times_{G_\beta} Y_\beta)) \sim \Cl_{\Gamma}(TU_\beta)\] 
induced by $\mathrm{Ind}_{G_\beta}^{G}[\bm{\mathcal{T}}_{\!\!\beta}] \in \KK^{G}(\Cl_{\Gamma}(T(G \times_{G_\beta} Y_\beta), \Cl_{\Gamma}(TU_\beta))$, which is also a Thom element (this can be seen via the properties of the induction functor). \\ 

After these preparations, we may now write (we remove the  algebra in subscript of the KK-products to alleviate the notations):
\[ j^G[\left.\sigma_{\!\nu}^\mathrm{tcl}\right|_{U_\beta}] \otimes [\overline{\partial}^{\mathrm{cl}}_{TU_\beta,\Gamma}] = j^G \, \mathrm{Ind}^{G}_{G_\beta} \big([\Lambda^{\bullet} (\g/\g_{\beta})] \boxtimes ([\left. \sigma_{\!\nu}^{\mathrm{tcl}} \right|_{V_\beta}]) \otimes [\mathcal{T}_\beta]^{-1})\big) \otimes  \underbrace{j^G \mathrm{Ind}_{G_\beta}^{G}[\bm{\mathcal{T}}_{\!\!\beta}]  \otimes [\overline{\partial}^{\mathrm{cl}}_{TU_\beta,\Gamma}]}_{
 \let\footnotesize
    \substack{\text{Dolbeault element }  [\overline{\partial}^{\mathrm{cl}}_{T(G \times_{G_\beta} Y_\beta),\Gamma}]}} \]
The equality highlighted in the bracket comes from a classical argument in KK-theory, see e.g \cite[Paragraph 5.4 (in the proof of the theorem)]{Kas88}. \\

To finish, the embedding $\iota_{Y_\beta, V_\beta}$ induces the homomorphism $\iota_{Y_\beta, V_\beta}^{*} : K_0(\Cl_{\Gamma}(TV_\beta)) \to K_0(\Cl_{\Gamma}(TY_\beta))$, and a straightforward adaptation of a classical formula of Atiyah--Singer gives
\[ [\mathcal{T}_\beta] \otimes_{\Cl_{\Gamma}(TV_\beta)} \iota_{Y_\beta, V_\beta}^{*}   = \textstyle \sum_k (-1)^k [ \Lambda^{k} \eta_\beta ] \]
Hence, $j^G[\left.\sigma_{\!\nu}^\mathrm{tcl}\right|_{U_\beta}] \otimes [\overline{\partial}^{\mathrm{cl}}_{TU_\beta,\Gamma}]$ is the transverse index of the symbol 
\[\mathrm{Ind}^G_{G_\beta} \dfrac{[ \Lambda^{\bullet} (\g/\g_\beta)] \boxtimes [\sigma|_{Y_\beta}]}{\sum_k (-1)^k [ \Lambda^{k} \eta_\beta ]} \in K_0\big(\Cl_{\Gamma}(T(G \times_{G_\beta} Y_\beta)\big). \]
Noticing that the normal bundle $N_\beta$ relative to the embedding $Y_\beta \hookrightarrow U_\beta$ decomposes as $N_\beta = \eta_\beta \oplus (\g/\g_{\beta} \times Y_\beta)$ finishes the proof.
\end{proof}

As a conclusion, we obtain the non-abelian localization formula of Paradan and Vergne:

\begin{theorem} \label{thm:main}
Let $(M,g)$ be a complete Riemannian manifold equipped with the action of a compact Lie group $G$, let $E$ be a $G$-equivariant Clifford module bundle over $M$ and let $\!\nu : M \to \g^{*}$ be a moment map in the sense of \emph{Definition \ref{def:moment map}}. According to the  decomposition of the zero set 
\[ \Zv = \bigsqcup_{\beta \in \mathcal{B}} G \cdot \big(M^{\beta} \cap \nu^{-1}(\beta) \big) \]
of the vector field $\!\bm{\nu}$ associated to $\!\nu$, where $\mathcal{B} \subset \g^*$ is a finite set consisting of representatives of the coadjoint orbits within $\!\nu(Z_{\!\bm{\nu}})$, the transverse index of the symbol $\sigma_{\!\nu}(x,\xi) = \i c(\xi + \!\bm{\nu}) \in C^{\infty}(T^*M, \pi_{T^*M}^*\End(E))$
is given by the fixed point formula:
\[ \mathrm{Index}_G [\sigma_{\!\nu}] = \sum_{\pi \in \widehat{G}} \big([E_0 \otimes (\nu^{-1}(0) \times_G V_\pi^*)] \otimes_{C(M_0)}[D_0] \big) \,  V_\pi + 
\underset{\beta \in \mathcal{B}}{\sum_{\beta \neq 0}}  \,\mathrm{Index}_{G_\beta} \dfrac{[ \Lambda^{\bullet}_\bC (\mathfrak{g}/\g_\beta)] \boxtimes [\left.\sigma_{\!\nu}\right|_{Y_\beta}]}{\sum_k (-1)^k [ \Lambda^{k} N_\beta ]} \]
where $E_0$ is the Clifford module bundle over the reduced space $M_0=\!\nu^{-1}(0)/G$ induced by $E$ ; $[E_0 \otimes (\nu^{-1}(0) \times_G V_\pi)]$ is the K-theory class of the vector bundle $E_0 \otimes (\nu^{-1}(0) \times_G V_\pi) \to M_0$ ; $G_\beta \subset G$ denotes the stabilizer subgroup of $\beta \in \mathcal{B}$ relative to the coadjoint action of $G$ on $\mathfrak{g}^*$ ; $\g_\beta = \mathrm{Lie}(G_\beta)$ ; $Y_\beta$ is a small open neighborhood of $M^\beta \cap \nu^{-1}(\beta)$ inside the fixed point set $M^\beta$ ; $[\sigma_{\!\nu}] \in K^0_G(T_G^*M)$; $[\left.\sigma_{\!\nu}\right|_{Y_\beta}] \in K^0_{G_\beta}(T_{G_\beta}Y_\beta)$ denote respectively the symbol classes of $\sigma_{\!\nu}$ and of its restriction to $T^*Y_\beta$; and $N_\beta$ is the normal bundle of $M^{\beta}$ in $M$.
\end{theorem}

The quantization commutes with reduction is subsequently proved without much difficulty with a separate calculation by hand, showing that only the term associated to the trivial representation in the first sum contributes when extracting the $G$-invariant part, cf. \cite{PV14Wit}. \\

\section{Concluding comments}

\subsection{Comparison to Paradan--Vergne's work} Paradan--Vergne's proof of Theorem \ref{thm:main} is actually quite KK-theoretical in spirit, but relies the K-theoretical machinery of Atiyah--Singer \cite{AtiSinI}. Let us just point out mention some improvements our approach brings. \\

Applying this theorem to the case of a toral action of $M=\bC^n$ with moment map $\nu(z) = (|z_1|^2/2, \ldots, |z_n|^2/2)$, one recovers directly the results of Atiyah \cite[Theorem 6.6] {AtiTransEll}, in contrast to Paradan--Vergne's work which uses this result (which is quite non-trivial!) as an intermediary step in the construction of their Thom isomorphism (that they refer to as push-forward). As a byproduct, we can avoid the technical constructions from \cite[Sections 4 and 6.3]{PV14Wit} utilized to this end, which are meant to overcome the non-availibility of the KK-product in Atiyah--Singer's framework. \\

By way of comparison, a similar phenomenon already appears when analyzing back to back the original proof of the classical theorem \cite{AtiSinI} via the `index axioms' (and notably the one on multiplicativity) and the KK-theoretical proof: the former suffers from complications which are subsequently removed after the introduction of the KK-product.  %For instance, in the former, the main source of difficulty is the verification by hand that the analytic index verifies the multiplicativity axiom. This is 

\vspace{2mm}

\subsection{Comparison with the analytic approach}
In \cite{LRS19}, it is shown that the class $j^G [\!\nu] \otimes [D]$ mentioned briefly in Section \ref{sec:moment map} corresponds to the K-homology class $[D_{f \!\nu}] \in K^0(C^*(G))$ of a deformed Dirac operator of the form
\[ D_{f \!\nu} = D + \i f c(\!\bm{\nu}) \]
where $f$ is a smooth $G$-invariant positive function satisfying a certain growth at infinity. \\

The crux of Ma--Tian--Zhang's solutions \cite{TZ98, MZ14} to  quantization commutes with reduction is that the localization the index of $D_{f \!\nu}$ around the zero set $\!\nu^{-1}(0)$ (and not only $\Zv$ as in Proposition \ref{prop: localization 1}) can be performed directly at the analytic level, at the cost of very technical estimates on the square of $D_{f \!\nu}$. \\

Hochs--Song's methods \cite{HS17} blend the analytic approach within Paradan--Vergne's framework: the idea is also to start from the decomposition of the zero set $Z_{\!\bm{\nu}} = \bigsqcup \, G \cdot (M^\beta \cap \nu^{-1}(\beta))$. Then, one analyzes the quantities 
\[ j^G[\left.\sigma_{\!\nu}^\mathrm{tcl}\right|_{U_\beta}] \otimes_{G\ltimes \Cl_\Gamma(TU_\beta)}[\overline{\partial}^{\mathrm{cl}}_{TU_\beta,\Gamma}] = j^G[\!\nu|_{U_\beta}] \otimes [D|_{U_\beta}] \]
directly from the defomed Dirac operator by establishing the Ma--Tian--Zhang estimates in slices around the different components of $\Zv$. \\

The present article combined with \cite{LRS19} hopefully highlights that these approaches are similar up to a KK-theoretical Poincar\'{e} duality. To make this analogy more exact, it might be interesting (or not) to see whether working out the analytic approach within the K-homological side $j^G[\nu] \otimes [D]$ leads to simplifications parallel to the calculations made in previous subsection.

%%%%%%%%%%%%%%%%%%%%%%%%%%%%%%%%%%%%%%%%
%%%%%%%%%%%% References %%%%%%%%%%%%%%%

\providecommand{\bysame}{\leavevmode\hbox to3em{\hrulefill}\thinspace}
\providecommand{\MR}{\relax\ifhmode\unskip\space\fi MR }
\providecommand{\MRhref}[2]{%
  \href{http://www.ams.org/mathscinet-getitem?mr=#1}{#2}
}
\providecommand{\href}[2]{#2}

\end{document}